\documentclass[11pt,a4paper,dvipsnames]{article}
\usepackage[
  a4paper,
  left=32mm,
  right=32mm,
  top=25mm,
  bottom=30mm
]{geometry}
\usepackage[T1]{fontenc}
\usepackage[english]{babel}
\usepackage{newtxtext}
\usepackage{microtype}
\usepackage{amsthm,amsmath,amssymb,mathtools}
\usepackage{newtxmath}
\numberwithin{equation}{section}
\newcommand{\set}[1]{\mathord{\left.\{ #1 \} \right. }}
\newcommand{\Bigset}[1]{\mathord{\left\{#1\right\}}}
\newcommand{\p}[1]{\mathord{(#1)}}
\newcommand{\Bigp}[1]{\mathord{\left(#1\right)}}
\newcommand{\abs}[1]{\left\lvert#1\right\rvert}
\newcommand{\Bignorm}[1]{\left\lVert#1\right\rVert}
\newcommand{\Biginner}[2]{\left\langle #1, #2 \right\rangle}
\newcommand{\inner}[2]{\langle #1, #2 \rangle}
\newcommand{\xmiddle}[1]{\;\middle#1\;}

\newcommand{\reals}{\mathbb{R}}
\newcommand{\naturals}{\mathbb{N}_0}
\newcommand{\calH}{\mathcal{H}}
\newcommand{\calG}{\mathcal{G}}

\DeclareMathOperator{\prox}{prox}
\DeclareMathOperator*{\minimize}{minimize}
\DeclareMathOperator*{\maximize}{maximize}
\DeclareMathOperator*{\argmin}{argmin}
\DeclareMathOperator{\dom}{dom}
\DeclareMathOperator{\Id}{Id}

\usepackage{enumitem}
\usepackage[bf,sf]{titlesec}

\usepackage{xcolor}
\definecolor{ForestGreen}{rgb}{0.13,0.55,0.13}
\definecolor{MidnightBlue}{rgb}{0.098,0.098,0.439}

\usepackage{xurl}
\usepackage{hyperref}
\hypersetup{
  colorlinks = true,
  urlcolor = MidnightBlue,
  linkcolor = MidnightBlue,
  citecolor = ForestGreen,
  pdftitle = {The Chambolle--Pock method converges weakly with 0 < theta <= 1/2 and tau sigma norm(L) squared < 4 theta (2 - theta)/(1 - 2 theta + 9 theta squared - 4 theta cubed)},
  pdfauthor = {Manu Upadhyaya},
  pdfsubject = {Weak convergence of the Chambolle--Pock method in real Hilbert spaces},
  pdfkeywords = {Chambolle--Pock method, PDHG, weak convergence, Lyapunov analysis}
}
\usepackage{bookmark}
\usepackage[capitalize,nameinlink]{cleveref}

\providecommand{\keywords}[1]{\textbf{\textsf{Keywords.}} #1}
\providecommand{\MSC}[1]{\textbf{\textsf{Mathematics subject classification 2020.}} #1}

\makeatletter
\renewenvironment{abstract}{%
  \if@twocolumn
    \section*{{\sffamily\bfseries\abstractname}}%
  \else
    \small
    \begin{center}%
      {\sffamily\bfseries\abstractname\vspace{-.5em}\vspace{\z@}}%
    \end{center}%
    \begin{quote}%
      \setlength{\parindent}{0pt}%
      \noindent\ignorespaces
  \fi
}{\if@twocolumn\else\end{quote}\fi}
\makeatother

\newtheoremstyle{manuscript}%
  {0.3cm}% space above
  {0.3cm}% space below
  {\itshape}% body font
  {}% indentation
  {}% head font; supplied below
  {:}% punctuation after head
  {.5em}% space after head
  {%
    \thmname{\normalfont\sffamily\bfseries #1 }%
    \thmnumber{\normalfont\sffamily\bfseries #2}%
    \thmnote{\normalfont\sffamily\bfseries\ (#3)}%
  }

\theoremstyle{manuscript}
\newtheorem{theorem}{Theorem}[section]
\newtheorem{assumption}[theorem]{Assumption}
\newtheorem{lemma}[theorem]{Lemma}
\newtheorem{definition}[theorem]{Definition}
\newtheorem{proposition}[theorem]{Proposition}

\AddToHook{env/theorem/begin}{\crefalias{section}{theorem}}
\AddToHook{env/assumption/begin}{\crefalias{theorem}{assumption}}
\AddToHook{env/lemma/begin}{\crefalias{theorem}{lemma}}
\AddToHook{env/definition/begin}{\crefalias{theorem}{definition}}
\AddToHook{env/proposition/begin}{\crefalias{theorem}{proposition}}

\newlist{enumeratass}{enumerate}{1}
\setlist[enumeratass]{label=(\roman*),ref=\theassumption(\roman*)}
\newlist{enumeratlem}{enumerate}{1}
\setlist[enumeratlem]{label=(\roman*),ref=\thelemma(\roman*)}

\crefname{theorem}{Theorem}{Theorems}
\crefname{assumption}{Assumption}{Assumptions}
\crefname{enumeratassi}{Assumption}{Assumptions}
\crefname{lemma}{Lemma}{Lemmas}
\crefname{enumeratlemi}{Lemma}{Lemmas}
\crefname{definition}{Definition}{Definitions}
\crefname{proposition}{Proposition}{Propositions}

\title{\sffamily\bfseries The Chambolle--Pock method converges weakly with \(0 < \theta \le 1/2\) and \(\tau\sigma\|L\|^{2} < 4\theta(2-\theta)/(1 - 2\theta + 9\theta^{2} - 4\theta^{3})\) \\[2ex]}

\author{
    Manu Upadhyaya
    \\[2ex]
}

\date{%
    Inria, D.I. ENS, CNRS, PSL Research University, Paris, France \\
    CMAP, École Polytechnique, Institut Polytechnique de Paris, Palaiseau, France \\
    \texttt{\href{mailto:manu.upadhyaya@inria.fr}{manu.upadhyaya@inria.fr}}
}

\begin{document}

\maketitle

\begin{abstract}
    The Chambolle--Pock method, also known as the primal-dual hybrid gradient method, is a standard first-order algorithm for convex-concave saddle-point problems and composite convex optimization.
    We establish weak sequential convergence of its primal-dual iterates in real Hilbert spaces for every \(0<\theta\leq 1/2\) whenever \(\tau\sigma\|L\|^{2}<4\theta(2-\theta)/(1-2\theta+9\theta^{2}-4\theta^{3})\).
    This extends the weak-convergence theory to a previously unexplored range of extrapolation parameters.
\end{abstract}


\MSC{
    47J25, % Iterative procedures involving nonlinear operators
    49M29, % Numerical methods involving duality
    65K05, % Numerical mathematical programming methods
    90C25, % Convex programming
    93D30  % Lyapunov and storage functions
}

\section{Introduction}\label{sec:introduction}

The Chambolle--Pock method~\cite{chambolle_pock_2011_first_order_primal}, also known as the primal-dual hybrid gradient (PDHG) method,
solves convex-concave saddle-point problems of the form
\begin{equation}\label{eq:saddle_point_problem}
    \minimize_{x\in\calH}
    \maximize_{y\in\calG}\;
    f\Bigp{x} + \Biginner{Lx}{y} - g^*\Bigp{y},
\end{equation}
where \(g^*\) denotes the convex conjugate of \(g\).
This is a primal-dual formulation of the primal composite optimization problem
\begin{equation}\label{eq:primal_composite_problem}
    \minimize_{x\in\calH}\; f\Bigp{x} + g\Bigp{Lx}.
\end{equation}

Throughout this paper, we make the following assumptions.
\begin{assumption}
    \label{ass:main}
    \begin{enumeratass}
        \item \label{ass:hilbert_spaces} \(\calH\) and \(\calG\) are real Hilbert spaces.
        \item \label{ass:f} The function \(f:\calH\to\reals\cup\{+\infty\}\) is convex, proper, and lower semicontinuous.
        \item \label{ass:g} The function \(g:\calG\to\reals\cup\{+\infty\}\) is convex, proper, and lower semicontinuous.
        \item \label{ass:L} The operator \(L:\calH\to\calG\) is a bounded linear operator.
        \item \label{ass:KKT} There exists a point \(\Bigp{x^{\star},y^{\star}}\in\calH\times\calG\) such that
                \begin{equation}\label{eq:KKT_conditions}
                    \begin{aligned}
                        -L^{*}y^{\star} &\in \partial f\Bigp{x^{\star}}, \\
                        Lx^{\star} &\in \partial g^{*}\Bigp{y^{\star}}.
                    \end{aligned}
                \end{equation}
                We call such points \emph{KKT points}.
    \end{enumeratass}
\end{assumption}
Every KKT point \(\Bigp{x^{\star},y^{\star}}\) is a saddle point of \eqref{eq:saddle_point_problem}, with \(x^{\star}\) solving \eqref{eq:primal_composite_problem}. In particular, the Chambolle--Pock method searches for such KKT points by iterating
\begin{equation}\label{eq:chambolle_pock_iteration}
    \Bigp{\forall k\in\naturals}\quad
    \left[
    \begin{aligned}
        x^{k+1} &= \prox_{\tau f}\Bigp{x^{k} - \tau L^{*}y^{k}}, \\
        y^{k+1} &= \prox_{\sigma g^*}\Bigp{y^{k} + \sigma L\Bigp{x^{k+1} + \theta\Bigp{x^{k+1} - x^{k}}}},
    \end{aligned}
    \right.
\end{equation}
for some primal and dual step sizes \(\tau,\sigma\in\reals_{++}\), an extrapolation parameter \(\theta\in\reals\), and an initial point \(\Bigp{x^{0},y^{0}}\in\calH\times\calG\).

The original convergence result of Chambolle and Pock~\cite{chambolle_pock_2011_first_order_primal} treats the case \(\theta=1\) under the condition \(\tau\sigma\Bignorm{L}^{2}<1\).
For \(\theta=1\), the method was subsequently interpreted as a preconditioned proximal point method in~\cite{he_yuan_2012_convergence_analysis}.
Still for \(\theta=1\), sequence convergence under \(\tau\sigma\Bignorm{L}^{2}\leq 1\) was established using a modified proximal point analysis~\cite{condat_2013_primal_dual_splitting}, while later finite-dimensional results enlarged the admissible range to \(\tau\sigma\Bignorm{L}^{2}<4/3\)~\cite{ma_etal_2023_preconditioned_pdhg,yan_li_2024_improved_conditions}.
The most closely related result to the present paper is~\cite{banert_etal_2026_chambolle_pock_converges}, where weak convergence in Hilbert spaces is established for the regime \(\theta>1/2\) under the condition \(\tau\sigma\Bignorm{L}^{2}<4/(1+2\theta)\).

The main contribution of the present paper is weak sequential convergence of the Chambolle--Pock iterates for every \(0<\theta\le 1/2\) under the step-size condition stated below, a regime where weak convergence had not previously been established.

\begin{theorem}[Weak sequential convergence]
    \label{thm:main:sequence_convergence}
    Suppose that \Cref{ass:main} holds.
    Let \(\Bigp{\Bigp{x^{k},y^{k}}}_{k\in\naturals}\) be generated by \eqref{eq:chambolle_pock_iteration} with
    \begin{align}\label{eq:main:sequence_convergence:parameter_condition}
        \tau,\sigma > 0,\quad 0 < \theta \leq 1/2, \quad\text{and}\quad \tau\sigma\Bignorm{L}^{2} < \frac{4\theta(2-\theta)}{1 - 2\theta + 9\theta^{2} - 4\theta^{3}},
    \end{align}
    from an arbitrary initial point \(\Bigp{x^{0},y^{0}}\in\calH\times\calG\).
    Then there exists a KKT point \(\Bigp{x^{\star},y^{\star}}\in\calH\times\calG\) such that
    \begin{align*}
        \Bigp{x^{k}, y^{k}} \rightharpoonup \Bigp{x^{\star}, y^{\star}}.
    \end{align*}
\end{theorem}

The proof of \Cref{thm:main:sequence_convergence} is given in \Cref{sec:proof_sequence_convergence}.

To the best of our knowledge, combining \Cref{thm:main:sequence_convergence} with~\cite[Theorem~2]{banert_etal_2026_chambolle_pock_converges} yields the largest currently known parameter region guaranteeing weak convergence of the unmodified Chambolle--Pock method over this problem class:
\begin{align*}
    \tau,\sigma>0
    \qquad\text{and}\qquad
    \tau\sigma\Bignorm{L}^{2}
    <
    \begin{cases}
        \displaystyle
        \frac{4\theta(2-\theta)}
        {1-2\theta+9\theta^{2}-4\theta^{3}},
        & \text{if }0<\theta\leq \tfrac{1}{2}, \\[1ex]
        \displaystyle
        \frac{4}{1+2\theta},
        & \text{if }\theta>\tfrac{1}{2}.
    \end{cases}
\end{align*}
Moreover, \(4/(1+2\theta)\) is a universal upper bound on admissible values of \(\tau\sigma\Bignorm{L}^{2}\) over this problem class for every \(\theta>0\). Indeed, for every such \(\theta\), the construction in \cite[Section~7]{banert_etal_2026_chambolle_pock_converges} gives an instance for which the iterates fail to converge whenever \(\tau\sigma\Bignorm{L}^{2}\geq 4/(1+2\theta)\).
The same example also shows that, when \(\theta\leq 0\), no positive choice of \(\tau\) and \(\sigma\) guarantees convergence from every initial point.

The present paper is also connected to recent computer-assisted Lyapunov analysis tools that provide numerical convergence certificates by solving specific semidefinite programs.
In particular, the Lyapunov analysis presented here was partially obtained using the \texttt{AutoLyap} software suite~\cite{upadhyaya_etal_2026_auto_lyap_software_suite}, which builds on the computer-assisted methodology of~\cite{upadhyaya_etal_2025_automated_tight_lyapunov}.

Beyond its classical role in imaging and inverse problems~\cite{chambolle_pock_2016_introduction_continuous_optimization}, the Chambolle--Pock method, or PDHG, has recently become a central primitive in large-scale linear programming.
In particular, the PDLP line of work derives practical LP solvers from PDHG and enriches the basic iteration with diagonal preconditioning, adaptive step sizes, restart schemes, infeasibility detection, and hardware-conscious implementations; see, for example,~\cite{applegate_etal_2021_practical_large_scale_lp,applegate_etal_2024_infeasibility_detection_pdhg_lp,applegate_etal_2023_faster_first_order_primal_dual_lp,lu_yang_2025_cupdlp}.
These developments make it especially relevant to understand the behavior of the underlying unmodified algorithm.
Our result contributes at this foundational level: it establishes convergence of the basic Chambolle--Pock/PDHG iteration itself, without restart or additional correction mechanisms, in a parameter regime that had not previously been covered.

\subsection{Notation and preliminaries}\label{sec:notation_preliminaries}

Let \(\naturals\) denote the set of nonnegative integers, \(\reals\) the set of real numbers, \(\reals_+\) the set of nonnegative real numbers, and \(\reals_{++}\) the set of positive real numbers.

For a function \(f:\calH\to\reals\cup\{+\infty\}\), its \emph{effective domain} is \(\dom f=\{x\in\calH\mid f(x)<+\infty\}\), and \(f\) is \emph{proper} if \(\dom f\neq\emptyset\). If \(f\) is proper, its \emph{subdifferential} is the set-valued operator \(\partial f:\calH\to 2^{\calH}\) defined, for each \(x\in\calH\), by
\begin{equation}\label{eq:subdifferential_definition}
    \partial f(x)=\Bigset{u\in\calH\xmiddle|f(y)\geq f(x)+\Biginner{u}{y-x}\text{ for each }y\in\calH}.
\end{equation}
The function \(f\) is \emph{convex} if \(f((1-\lambda)x+\lambda y)\leq (1-\lambda)f(x)+\lambda f(y)\) for each \(x,y\in\calH\) and \(\lambda\in[0,1]\), and it is \emph{lower semicontinuous} if \(\liminf_{y\to x}f(y)\geq f(x)\) for each \(x\in\calH\).

Suppose that \(f\) is proper, convex, and lower semicontinuous, and let \(\gamma\in\reals_{++}\). The \emph{proximal operator} of \(\gamma f\) is the single-valued mapping \(\prox_{\gamma f}:\calH\to\calH\) given by \(\prox_{\gamma f}(x)=\argmin_{z\in\calH}\bigl(f(z)+(2\gamma)^{-1}\Bignorm{x-z}^{2}\bigr)\). Moreover, for each \(x\in\calH\),
\begin{align}\label{eq:prox_finite}
    f\Bigp{\prox_{\gamma f}\Bigp{x}} < + \infty.
\end{align}
See~\cite[Proposition 12.15]{bauschke_combettes_2017_convex_analysis_monotone}. The \emph{convex conjugate} of \(f\) is the function \(f^*:\calH\to\reals\cup\{+\infty\}\) given by \(f^*(u)=\sup_{x\in\calH}\bigl(\Biginner{u}{x}-f(x)\bigr)\) for each \(u\in\calH\).
We will use the following standard properties. By \cite[Propositions 16.6 and 16.44]{bauschke_combettes_2017_convex_analysis_monotone}, for each \(x,p\in\calH\),
\begin{equation}\label{eq:prox_characterization}
    p=\prox_{\gamma f}(x)
    \quad\Longleftrightarrow\quad
    \gamma^{-1}(x-p)\in\partial f(p).
\end{equation}
Moreover, by \cite[Corollary 13.38]{bauschke_combettes_2017_convex_analysis_monotone}, \(f^{*}\) is proper, convex, and lower semicontinuous, and by \cite[Proposition 16.16]{bauschke_combettes_2017_convex_analysis_monotone},
\(u\in\partial f(x)\) if and only if \(x\in\partial f^*(u)\), for each \(x,u\in\calH\).

An operator \(L:\calH\to\calG\) is \emph{linear} if \(L(\alpha x+\beta y)=\alpha Lx+\beta Ly\) for each \(x,y\in\calH\) and \(\alpha,\beta\in\reals\). It is \emph{bounded} if there exists \(M\in\reals_+\) such that \(\Bignorm{Lx}\leq M\Bignorm{x}\) for each \(x\in\calH\); the smallest such constant is the \emph{operator norm} \(\Bignorm{L}\).
If \(L\) is linear and bounded, its \emph{adjoint} is the unique bounded linear operator \(L^*:\calG\to\calH\) satisfying \(\Biginner{Lx}{y}=\Biginner{x}{L^*y}\) for each \(x\in\calH\) and \(y\in\calG\). When \(\calH=\calG\), the operator \(L\) is \emph{self-adjoint} if \(L=L^*\), \emph{positive} if \(\Biginner{Lx}{x}\geq0\) for each \(x\in\calH\), and \emph{strongly positive} if there exists \(\mu\in\reals_{++}\) such that \(\Biginner{Lx}{x}\geq\mu\Bignorm{x}^2\) for each \(x\in\calH\).

We recall the Cauchy--Schwarz inequality \(\abs{\Biginner{x}{y}}\leq\Bignorm{x}\Bignorm{y}\) and Young's inequality \(2\Biginner{x}{y}\leq\alpha\Bignorm{x}^{2}+\alpha^{-1}\Bignorm{y}^{2}\), valid for each \(x,y\in\calH\) and \(\alpha\in\reals_{++}\).
On \(\calH\times\calG\), we use the canonical inner product \(\Biginner{(x,y)}{(\bar x,\bar y)}=\Biginner{x}{\bar x}+\Biginner{y}{\bar y}\) and denote its induced norm by \(\Bignorm{\cdot}\).

\subsection{Organization}

The remainder of this manuscript is organized as follows. \Cref{sec:auxiliary_inequalities} establishes the auxiliary inequalities and metric properties needed for the analysis. \Cref{sec:lyapunov_analysis} develops the Lyapunov analysis. \Cref{sec:proof_sequence_convergence} proves \Cref{thm:main:sequence_convergence}. We conclude in \Cref{sec:conclusion}.

\section{Auxiliary inequalities and metric properties}\label{sec:auxiliary_inequalities}

This section establishes the parameter inequalities and metric properties used in the Lyapunov analysis of \Cref{sec:lyapunov_analysis}. These auxiliary results and the Lyapunov estimates remain valid throughout the interval \(0<\theta\leq 1\) under the non-strict step-size condition
\begin{align}\label{eq:lyapunov:parameter_condition}
    \tau,\sigma > 0,\quad 0 < \theta \leq 1, \quad\text{and}\quad \tau\sigma\Bignorm{L}^{2} \leq \frac{4\theta(2-\theta)}{1 - 2\theta + 9\theta^{2} - 4\theta^{3}}.
\end{align}
The denominator in \eqref{eq:lyapunov:parameter_condition} is positive, since \(1-2\theta+9\theta^{2}-4\theta^{3}=(1-\theta)^{2}+4\theta^{2}(2-\theta)>0\) for \(0<\theta\leq 1\).
Whenever strictness is required, we assume that the last inequality in \eqref{eq:lyapunov:parameter_condition} is strict.
Under this assumption, the argument in \Cref{sec:proof_sequence_convergence} also yields weak convergence throughout \(0<\theta\leq 1\).
For \(\theta>1/2\), however,~\cite[Theorem~2]{banert_etal_2026_chambolle_pock_converges} permits the larger step-size bound \(\tau\sigma\Bignorm{L}^{2}<4/(1+2\theta)\).
We therefore state \Cref{thm:main:sequence_convergence} only for the previously uncovered range \(0<\theta\leq 1/2\).

\begin{lemma}
    \label{lem:parameters}
    The parameter condition \eqref{eq:lyapunov:parameter_condition} implies that
    \begin{align}\label{eq:inequality_lemma:1}
        4\theta(2-\theta) - \tau\sigma\Bignorm{L}^{2}\Bigp{1 - 2\theta + 9\theta^{2} - 4\theta^{3}} \geq 0
    \end{align}
    and
    \begin{align}\label{eq:inequality_lemma:2}
        \tau\sigma\Bignorm{L}^{2} (1+\theta)^{2} \leq 4.
    \end{align}
    Similarly, if the last inequality in \eqref{eq:lyapunov:parameter_condition} is strict, then the inequalities in \eqref{eq:inequality_lemma:1} and \eqref{eq:inequality_lemma:2} are strict.
\end{lemma}

\begin{proof}
    The first inequality \eqref{eq:inequality_lemma:1} follows by multiplying both sides of the last inequality in \eqref{eq:lyapunov:parameter_condition} by its positive denominator and rearranging. For the second inequality, note that
    \begin{align*}
        \frac{4\theta(2-\theta)}{1 - 2\theta + 9\theta^{2} - 4\theta^{3}}
        \le
        \frac{4}{(1+\theta)^{2}}
    \end{align*}
    is equivalent to \((1-\theta)^{4}\ge 0\), which gives \eqref{eq:inequality_lemma:2}.
\end{proof}

\begin{definition}
    \label{def:P}
    Suppose that \Cref{ass:hilbert_spaces,ass:L} hold, and let \(\tau,\sigma\in\reals_{++}\) and \(\theta\in\reals\).
    Define the bounded linear operator
    \begin{align*}
        P
        =
        \begin{bmatrix}
            \frac{1}{\tau}\Id & -\frac{1+\theta}{2}L^{*} \\[1ex]
            -\frac{1+\theta}{2}L & \frac{1}{\sigma}\Id
        \end{bmatrix},
    \end{align*}
    on \(\calH\times\calG\), the bilinear form
    \begin{align*}
        \Bigp{\forall z_1, z_2\in\calH\times\calG}\quad \Biginner{z_1}{z_2}_{P} = \Biginner{z_1}{Pz_2},
    \end{align*}
    and the associated quadratic form
    \begin{align*}
        \Bigp{\forall (x,y)\in\calH\times\calG}\quad
        \Bignorm{(x,y)}_{P}^{2}
        =
        \Biginner{(x,y)}{(x,y)}_{P}
        =
        \frac{1}{\tau}\Bignorm{x}^{2}
        + \frac{1}{\sigma}\Bignorm{y}^{2}
        - (1+\theta)\Biginner{Lx}{y}.
    \end{align*}
\end{definition}

\begin{lemma}
    \label{lem:P_properties}
    Suppose that \Cref{ass:hilbert_spaces,ass:L} hold, and let \(\tau,\sigma\in\reals_{++}\) and \(\theta\in\reals\).
    \begin{enumeratlem}
        \item\label{lem:P_properties:self_adjoint}
        The operator \(P\) from \Cref{def:P} is self-adjoint.
        \item\label{lem:P_properties:positive}
        If \eqref{eq:lyapunov:parameter_condition} holds, then \(P\) is positive on \(\calH\times\calG\) and \(\Bignorm{\cdot}_{P}\) is a seminorm on \(\calH\times\calG\).
        \item\label{lem:P_properties:strict}
        If the last inequality in \eqref{eq:lyapunov:parameter_condition} is strict, then \(P\) is strongly positive on \(\calH\times\calG\).
        Consequently, \(P\) is bijective, \(\Biginner{\cdot}{\cdot}_{P}\) is an inner product on \(\calH\times\calG\), and \(\Bignorm{\cdot}_{P}\) is a norm equivalent to the canonical product norm \(\Bignorm{\cdot}\) on \(\calH\times\calG\).
        In particular, the Hilbert spaces \(\Bigp{\calH\times\calG, \inner{\cdot}{\cdot}}\) and \(\Bigp{\calH\times\calG, \inner{\cdot}{\cdot}_{P}}\) have the same weakly convergent sequences and their corresponding weak limits are the same in both spaces.
    \end{enumeratlem}
\end{lemma}

\begin{proof}
    \leavevmode
    \begin{itemize}[leftmargin=5em,labelwidth=4em,labelsep=0.5em]
        \item[\ref{lem:P_properties:self_adjoint}.]
        This is immediate from the block representation of \(P\).

        \item[\ref{lem:P_properties:positive}.]
        Let \((x,y)\in\calH\times\calG\). By Cauchy--Schwarz and Young's inequality,
        \begin{align*}
            (1+\theta)\abs{\Biginner{Lx}{y}}
            \le
            (1+\theta)\Bignorm{L}\Bignorm{x}\Bignorm{y}
            \le
            \frac{\sqrt{\tau\sigma}(1+\theta)\Bignorm{L}}{2}
            \Bigp{
                \frac{1}{\tau}\Bignorm{x}^{2} + \frac{1}{\sigma}\Bignorm{y}^{2}
            }.
        \end{align*}
        Therefore,
        \begin{equation}\label{eq:two_sided_estimate_P}
            \begin{aligned}
                & \Bigp{
                    1 - \frac{\sqrt{\tau\sigma}(1+\theta)\Bignorm{L}}{2}
                }
                \min\Bigp{\frac{1}{\tau},\frac{1}{\sigma}}
                \Bigp{
                    \Bignorm{x}^{2} + \Bignorm{y}^{2}
                } \\
                & \le \Bignorm{(x,y)}_{P}^{2} \\
                & \le
                \Bigp{
                    1 + \frac{\sqrt{\tau\sigma}(1+\theta)\Bignorm{L}}{2}
                }
                \max\Bigp{\frac{1}{\tau},\frac{1}{\sigma}}
                \Bigp{
                    \Bignorm{x}^{2} + \Bignorm{y}^{2}
                }.
            \end{aligned}
        \end{equation}
        If \eqref{eq:lyapunov:parameter_condition} holds, then \eqref{eq:inequality_lemma:2} in \Cref{lem:parameters} gives
        \begin{align*}
            1 - \frac{\sqrt{\tau\sigma}(1+\theta)\Bignorm{L}}{2} \ge 0,
        \end{align*}
        and from \eqref{eq:two_sided_estimate_P} we conclude that
        \begin{align*}
            \Bignorm{(x,y)}_{P}^{2} \ge 0,
        \end{align*}
        i.e., \(P\) is positive and \(\Bignorm{\cdot}_{P}\) is a seminorm on \(\calH\times\calG\).

        \item[\ref{lem:P_properties:strict}.]
        If the last inequality in \eqref{eq:lyapunov:parameter_condition} is strict, then the strict version of \eqref{eq:inequality_lemma:2} in \Cref{lem:parameters} gives
        \begin{align*}
            1 - \frac{\sqrt{\tau\sigma}(1+\theta)\Bignorm{L}}{2} > 0,
        \end{align*}
        and from \eqref{eq:two_sided_estimate_P} we conclude that \(P\) is strongly positive, which implies injectivity.
        Since \(P\) is self-adjoint by \Cref{lem:P_properties:self_adjoint} and strongly positive, \(\Biginner{\cdot}{\cdot}_{P}\) is an inner product on \(\calH\times\calG\).
        Moreover, \eqref{eq:two_sided_estimate_P} shows that \(\Bignorm{\cdot}_{P}\) is equivalent to the canonical product norm on \(\calH\times\calG\).
        Next, the Lax--Milgram theorem \cite[Example 27.12]{bauschke_combettes_2017_convex_analysis_monotone} and the Riesz--Fr{\'e}chet representation theorem \cite[Fact 2.24]{bauschke_combettes_2017_convex_analysis_monotone} imply that \(P\) is surjective.
        Finally, let \(\p{z^k}_{k\in\naturals}\) be a sequence in \(\calH\times\calG\) and let \(z\in\calH\times\calG\).
        Since \(P\) is surjective, we get
        \begin{align*}
            \begin{aligned}
            z^k \rightharpoonup z \text{ in } (\calH\times\calG,\langle\cdot,\cdot\rangle_P)
            &\iff (\forall u\in \calH\times\calG)\ \langle z^k-z,u\rangle_P \to 0\\
            &\iff (\forall u\in \calH\times\calG)\ \langle z^k-z,Pu\rangle \to 0\\
            &\iff (\forall v\in \calH\times\calG)\ \langle z^k-z,v\rangle \to 0\\
            &\iff z^k \rightharpoonup z \text{ in } (\calH\times\calG,\langle\cdot,\cdot\rangle).
            \end{aligned}
        \end{align*}
        Therefore, the weak topologies induced by \(\langle\cdot,\cdot\rangle_P\) and by the canonical
        inner product \(\langle\cdot,\cdot\rangle\) coincide. In particular, the two Hilbert space structures
        have the same weakly convergent sequences, and a sequence has the same weak limit in both
        spaces.
    \end{itemize}
\end{proof}

\section{Lyapunov analysis}\label{sec:lyapunov_analysis}

Under \eqref{eq:lyapunov:parameter_condition}, we develop the Lyapunov estimates used to prove \Cref{thm:main:sequence_convergence}.
To formulate these estimates, let
\begin{align*}
    \mathcal{L}\Bigp{x,y} = f\Bigp{x} + \Biginner{y}{L x} - g^* \Bigp{y}
\end{align*}
denote the Lagrangian associated with \eqref{eq:saddle_point_problem}.
For a KKT point \(\Bigp{x^{\star}, y^{\star}} \in \calH \times \calG\), define the \emph{duality-gap function} \(\mathcal{D}_{x^{\star},y^{\star}}:\calH\times\calG \to \reals\cup\set{+\infty}\) by
\begin{align}\label{eq:pd_gap_fcn}
    \Bigp{\forall \Bigp{x,y}\in\calH\times\calG}\quad \mathcal{D}_{x^{\star},y^{\star}}\Bigp{x,y} = \mathcal{L}\Bigp{x,y^{\star}}-\mathcal{L}\Bigp{x^{\star},y}.
\end{align}
This function is finite on \(\dom f \times \dom g^{*}\) and nonnegative on \(\calH\times\calG\); in particular,
\begin{align}\label{eq:D_finite}
    \Bigp{\forall \Bigp{x,y} \in \dom f \times \dom g^{*}}\quad &\mathcal{D}_{x^{\star},y^{\star}}\Bigp{x,y} < +\infty,
    \\ \label{eq:D_nonnegative}
    \Bigp{\forall \Bigp{x,y} \in \calH \times \calG}\quad &\mathcal{D}_{x^{\star},y^{\star}}\Bigp{x,y} \geq 0.
\end{align}

\begin{definition}[Lyapunov function]
    \label{def:lyapunov_function}
    Suppose that \Cref{ass:main} holds.
    Let \(\Bigp{\Bigp{x^{k},y^{k}}}_{k\in\naturals}\) be generated by \eqref{eq:chambolle_pock_iteration}, and fix a KKT point \(\Bigp{x^{\star},y^{\star}}\in\calH\times\calG\). Let \(\mathcal{D}_{x^{\star},y^{\star}}\) be the duality gap function from \eqref{eq:pd_gap_fcn}, and let \(\Bignorm{\cdot}_{P}\) be the seminorm from \Cref{lem:P_properties:positive} when \eqref{eq:lyapunov:parameter_condition} holds and the norm from \Cref{lem:P_properties:strict} when the last inequality in \eqref{eq:lyapunov:parameter_condition} is strict.
    Define the \emph{Lyapunov function} \(\mathcal{V}:\naturals\to\reals\) by
    \begin{equation}\label{eq:lyapunov_function}
            \Bigp{\forall k \in \naturals} \quad \mathcal{V}(k) {}={}
            \begin{aligned}[t]
                &\frac{1}{2}\Bignorm{\Bigp{x^{k} - x^{\star}, y^{k} - y^{\star}}}_{P}^{2}
                - \frac{1}{4}\Bignorm{\Bigp{x^{k+1} - x^{k}, y^{k+1} - y^{k}}}_{P}^{2} \\
                &- \frac{1-\theta}{2} \mathcal{D}_{x^{\star},y^{\star}}\Bigp{x^{k+1},y^{k+1}} \\
                &-
                \frac{1-\theta}{2}
                \left(\vphantom{\Biginner{y^{k} - y^{\star}}{L\Bigp{x^{k+1} - x^{k}}}}\right.
                    \Biginner{y^{k} - y^{\star}}{L\Bigp{x^{k+1} - x^{k}}}
                    - \Biginner{L\Bigp{x^{k} - x^{\star}}}{y^{k+1} - y^{k}}
                \left.\vphantom{\Biginner{y^{k} - y^{\star}}{L\Bigp{x^{k+1} - x^{k}}}}\right).
            \end{aligned}
    \end{equation}
\end{definition}

\begin{proposition}[Lyapunov lower bound]
    \label{prop:lyapunov_lower_bound}
    Suppose that \Cref{ass:main} holds.\hfill\break
    Let \(\Bigp{\Bigp{x^{k},y^{k}}}_{k\in\naturals}\) be generated by \eqref{eq:chambolle_pock_iteration} with \eqref{eq:lyapunov:parameter_condition}, \(\Bigp{x^{\star},y^{\star}}\in\calH\times\calG\) satisfy \eqref{eq:KKT_conditions}, and \(\mathcal{V}\) be given by \Cref{def:lyapunov_function}.
    Then
    \begin{equation}\label{eq:lyapunov_lower_bound}
    \begin{aligned}
        \Bigp{\forall k\in\naturals}\quad
        \mathcal{V}(k)
        \geq
        \frac{1}{2}
        \Bignorm{\Bigp{x^{k+1} - x^{\star}, y^{k+1} - y^{\star}}}_{P}^{2}
        \geq 0.
    \end{aligned}
    \end{equation}
\end{proposition}

\begin{proof}
    By the characterization of the proximal operator in \eqref{eq:prox_characterization}, \eqref{eq:chambolle_pock_iteration} is equivalent to
    \begin{align*}
        \frac{1}{\tau}x^{k} - L^{*}y^{k} - \frac{1}{\tau}x^{k+1} &\in \partial f\Bigp{x^{k+1}}, \\
        - \theta Lx^{k} + \frac{1}{\sigma}y^{k} + (1+\theta)Lx^{k+1} - \frac{1}{\sigma}y^{k+1} &\in \partial g^*\Bigp{y^{k+1}}.
    \end{align*}
    Hence, by the definition of the subdifferential in \eqref{eq:subdifferential_definition} and \eqref{eq:KKT_conditions},
    \begingroup
    \allowdisplaybreaks
    \begin{align*}
        \mathcal{V}(k)
        \geq{}&
        \begin{aligned}[t]
            &\mathcal{V}(k) \\
            &+ \frac{1+\theta}{2}
            \underbracket{
                \Bigp{
                    f\Bigp{x^{\star}} - f\Bigp{x^{k+1}}
                    - \Biginner{y^{\star}}{Lx^{k+1} - Lx^{\star}}
                }
            }_{\le 0} \\
            &+ \underbracket{
                f\Bigp{x^{k+1}} - f\Bigp{x^{\star}}
                + \Biginner{\frac{1}{\tau}x^{k} - L^{*}y^{k} - \frac{1}{\tau}x^{k+1}}{x^{\star} - x^{k+1}}
            }_{\le 0} \\
            &+ \frac{1+\theta}{2}
            \underbracket{
                \Bigp{
                    g^*\Bigp{y^{\star}} - g^*\Bigp{y^{k+1}}
                    + \Biginner{Lx^{\star}}{y^{k+1} - y^{\star}}
                }
            }_{\le 0} \\
            &+ \underbracket{
                g^*\Bigp{y^{k+1}} - g^*\Bigp{y^{\star}}
                + \Biginner{-\theta Lx^{k} + \frac{1}{\sigma}y^{k} + (1+\theta)Lx^{k+1} - \frac{1}{\sigma}y^{k+1}}{y^{\star} - y^{k+1}}
            }_{\le 0}
        \end{aligned} \\
        ={}&
        \begin{aligned}[t]
            &\frac{1}{2}\Bignorm{\Bigp{x^{k} - x^{\star}, y^{k} - y^{\star}}}_{P}^{2}
            - \frac{1}{4}\Bignorm{\Bigp{x^{k+1} - x^{k}, y^{k+1} - y^{k}}}_{P}^{2} \\
            &- \frac{1-\theta}{2}
            \left(\vphantom{\Biginner{y^{\star}}{Lx^{k+1} - Lx^{\star}}}\right.
                f\Bigp{x^{k+1}} - f\Bigp{x^{\star}}
                + \Biginner{y^{\star}}{Lx^{k+1} - Lx^{\star}} \\
            &\hspace{4.5em}
                + g^*\Bigp{y^{k+1}} - g^*\Bigp{y^{\star}}
                - \Biginner{Lx^{\star}}{y^{k+1} - y^{\star}}
            \left.\vphantom{\Biginner{y^{\star}}{Lx^{k+1} - Lx^{\star}}}\right) \\
            &- \frac{1-\theta}{2}
            \left(\vphantom{\Biginner{y^{k} - y^{\star}}{L\Bigp{x^{k+1} - x^{k}}}}\right.
                \Biginner{y^{k} - y^{\star}}{L\Bigp{x^{k+1} - x^{k}}}
                - \Biginner{L\Bigp{x^{k} - x^{\star}}}{y^{k+1} - y^{k}}
            \left.\vphantom{\Biginner{y^{k} - y^{\star}}{L\Bigp{x^{k+1} - x^{k}}}}\right) \\
            &+ \frac{1+\theta}{2}
            \left(\vphantom{\Biginner{y^{\star}}{Lx^{k+1} - Lx^{\star}}}\right.
                f\Bigp{x^{\star}} - f\Bigp{x^{k+1}}
                - \Biginner{y^{\star}}{Lx^{k+1} - Lx^{\star}}
            \left.\vphantom{\Biginner{y^{\star}}{Lx^{k+1} - Lx^{\star}}}\right) \\
            &+ f\Bigp{x^{k+1}} - f\Bigp{x^{\star}}
            + \Biginner{\frac{1}{\tau}x^{k} - L^{*}y^{k} - \frac{1}{\tau}x^{k+1}}{x^{\star} - x^{k+1}} \\
            &+ \frac{1+\theta}{2}
            \left(\vphantom{\Biginner{Lx^{\star}}{y^{k+1} - y^{\star}}}\right.
                g^*\Bigp{y^{\star}} - g^*\Bigp{y^{k+1}}
                + \Biginner{Lx^{\star}}{y^{k+1} - y^{\star}}
            \left.\vphantom{\Biginner{Lx^{\star}}{y^{k+1} - y^{\star}}}\right) \\
            &+ g^*\Bigp{y^{k+1}} - g^*\Bigp{y^{\star}}
            + \Biginner{-\theta Lx^{k} + \frac{1}{\sigma}y^{k} + (1+\theta)Lx^{k+1} - \frac{1}{\sigma}y^{k+1}}{y^{\star} - y^{k+1}}.
        \end{aligned} \\
        ={}&
        \begin{aligned}[t]
            &\frac{1}{2}\Bignorm{\Bigp{x^{k} - x^{\star}, y^{k} - y^{\star}}}_{P}^{2}
            - \frac{1}{4}\Bignorm{\Bigp{x^{k+1} - x^{k}, y^{k+1} - y^{k}}}_{P}^{2} \\
            &- \frac{1-\theta}{2}
            \left(\vphantom{\Biginner{y^{k} - y^{\star}}{L\Bigp{x^{k+1} - x^{k}}}}\right.
                \Biginner{y^{k} - y^{\star}}{L\Bigp{x^{k+1} - x^{k}}}
                - \Biginner{L\Bigp{x^{k} - x^{\star}}}{y^{k+1} - y^{k}}
            \left.\vphantom{\Biginner{y^{k} - y^{\star}}{L\Bigp{x^{k+1} - x^{k}}}}\right) \\
            &- \Biginner{y^{\star}}{Lx^{k+1} - Lx^{\star}}
            + \Biginner{\frac{1}{\tau}x^{k} - L^{*}y^{k} - \frac{1}{\tau}x^{k+1}}{x^{\star} - x^{k+1}} \\
            &+ \Biginner{Lx^{\star}}{y^{k+1} - y^{\star}}
            + \Biginner{-\theta Lx^{k} + \frac{1}{\sigma}y^{k} + (1+\theta)Lx^{k+1} - \frac{1}{\sigma}y^{k+1}}{y^{\star} - y^{k+1}}.
        \end{aligned} \\
        ={}&
        \begin{aligned}[t]
            &\frac{1}{2\tau}\Bignorm{x^{k} - x^{\star}}^{2}
            - \frac{1}{4\tau}\Bignorm{x^{k+1} - x^{k}}^{2}
            + \frac{1}{\tau}\Biginner{x^{k+1} - x^{k}}{x^{k+1} - x^{\star}} \\
            &+ \frac{1}{2\sigma}\Bignorm{y^{k} - y^{\star}}^{2}
            - \frac{1}{4\sigma}\Bignorm{y^{k+1} - y^{k}}^{2}
            + \frac{1}{\sigma}\Biginner{y^{k+1} - y^{k}}{y^{k+1} - y^{\star}} \\
            &- \frac{1+\theta}{2}\Biginner{L\Bigp{x^{k} - x^{\star}}}{y^{k} - y^{\star}}
            + \frac{1+\theta}{4}\Biginner{L\Bigp{x^{k+1} - x^{k}}}{y^{k+1} - y^{k}} \\
            &- \frac{1-\theta}{2}\Biginner{y^{k} - y^{\star}}{L\Bigp{x^{k+1} - x^{k}}}
            + \frac{1-\theta}{2}\Biginner{L\Bigp{x^{k} - x^{\star}}}{y^{k+1} - y^{k}} \\
            &+ \Biginner{y^{k} - y^{\star}}{L\Bigp{x^{k+1} - x^{\star}}}
            + \Biginner{\theta Lx^{k} - (1+\theta)Lx^{k+1} + Lx^{\star}}{y^{k+1} - y^{\star}}.
        \end{aligned} \\
        ={}&
        \begin{aligned}[t]
            &\frac{1}{2\tau}\Bignorm{x^{k} - x^{\star}}^{2}
            - \frac{1}{4\tau}\Bignorm{x^{k+1} - x^{k}}^{2}
            + \frac{1}{\tau}\Biginner{x^{k+1} - x^{k}}{x^{k+1} - x^{\star}} \\
            &+ \frac{1}{2\sigma}\Bignorm{y^{k} - y^{\star}}^{2}
            - \frac{1}{4\sigma}\Bignorm{y^{k+1} - y^{k}}^{2}
            + \frac{1}{\sigma}\Biginner{y^{k+1} - y^{k}}{y^{k+1} - y^{\star}} \\
            &- \frac{1+\theta}{2}\Biginner{L\Bigp{x^{k} - x^{\star}}}{y^{k} - y^{\star}}
            - \frac{1+\theta}{2}\Biginner{L\Bigp{x^{k+1} - x^{k}}}{y^{k} - y^{\star}} \\
            &- \frac{1+\theta}{2}\Biginner{L\Bigp{x^{k} - x^{\star}}}{y^{k+1} - y^{k}}
            - \frac{3(1+\theta)}{4}\Biginner{L\Bigp{x^{k+1} - x^{k}}}{y^{k+1} - y^{k}}.
        \end{aligned} \\
        ={}&
        \begin{aligned}[t]
            &\frac{1}{4\tau}\Bignorm{x^{k+1} - x^{k}}^{2}
            + \frac{1}{2\tau}\Bignorm{x^{k+1} - x^{\star}}^{2} \\
            &+ \frac{1}{4\sigma}\Bignorm{y^{k+1} - y^{k}}^{2}
            + \frac{1}{2\sigma}\Bignorm{y^{k+1} - y^{\star}}^{2} \\
            &- \frac{1+\theta}{4}\Biginner{L\Bigp{x^{k+1} - x^{k}}}{y^{k+1} - y^{k}} \\
            &- \frac{1+\theta}{2}\Biginner{L\Bigp{x^{k+1} - x^{\star}}}{y^{k+1} - y^{\star}}.
        \end{aligned} \\
        ={}&
        \frac{1}{4}\Bignorm{\Bigp{x^{k+1} - x^{k}, y^{k+1} - y^{k}}}_{P}^{2}
        + \frac{1}{2}\Bignorm{\Bigp{x^{k+1} - x^{\star}, y^{k+1} - y^{\star}}}_{P}^{2}.
    \end{align*}
    \endgroup
    The first inequality uses only that \(0<\theta \leq 1\), so each underbraced term can be added without increasing the right-hand side. In the identities that follow, we first substitute the definitions of \(\mathcal{V}(k)\) and \(\mathcal{D}_{x^{\star},y^{\star}}\), then use \(\Biginner{L^{*}y}{x} = \Biginner{y}{Lx}\) together with \(\frac{1-\theta}{2} + \frac{1+\theta}{2} = 1\) to cancel the function-value terms. Finally, the quadratic terms in \(x\) and \(y\) are completed, and the remaining \(L\)-coupling terms are collected back into the two \(P\)-quadratic forms.
    Since \eqref{eq:lyapunov:parameter_condition} implies, by \Cref{lem:P_properties:positive}, that \(\Bignorm{\cdot}_{P}\) is a seminorm on \(\calH\times\calG\), both terms on the right-hand side are nonnegative.
\end{proof}

\begin{proposition}[Lyapunov inequality]
    \label{prop:lyapunov_inequality}
    Suppose that \Cref{ass:main} holds.
    Let \(\Bigp{\Bigp{x^{k},y^{k}}}_{k\in\naturals}\) be generated by \eqref{eq:chambolle_pock_iteration} with \eqref{eq:lyapunov:parameter_condition}, \(\Bigp{x^{\star},y^{\star}}\in\calH\times\calG\) satisfy \eqref{eq:KKT_conditions}, \(\mathcal{D}_{x^{\star},y^{\star}}\) be the duality gap function from \eqref{eq:pd_gap_fcn}, and \(\mathcal{V}\) be given by \Cref{def:lyapunov_function}.
    Define the bounded linear operator \(K:\calH\to\calG\) by
    \begin{equation*}
        K
        =
        \begin{cases}
            L / \Bignorm{L}, & \text{if } L \neq 0, \\
            0, & \text{if } L = 0,
        \end{cases}
    \end{equation*}
    and define
    \begin{equation}\label{eq:eta_coefficients_lyapunov_inequality}
    \begin{aligned}
        \eta_{\pm}
        :={}&
        \frac{
            4\theta(2-\theta) - \tau\sigma\Bignorm{L}^{2}\Bigp{1 - 2\theta + 9\theta^{2} - 4\theta^{3}}
        }{
            8\Bigp{1 \pm \sqrt{\tau\sigma}\Bignorm{L}\theta(1-\theta)}
        }.
    \end{aligned}
    \end{equation}
    Then the denominators in \eqref{eq:eta_coefficients_lyapunov_inequality} are positive and
    \begin{align}\label{eq:eta_coefficients_lyapunov_inequality:pos}
        \eta_{\pm}\ge 0.
    \end{align}
    Moreover,
    \begingroup
    \allowdisplaybreaks
    \begin{equation}\label{eq:lyapunov_inequality}
        \Bigp{\forall k\in\naturals}\quad
        \mathcal{V}(k+1) {}\leq{}
        \begin{aligned}[t]
            &\mathcal{V}(k)
            -
            \mathcal{D}_{x^{\star},y^{\star}}\Bigp{x^{k+1},y^{k+1}} \\
            &-
            \frac{\theta}{4\tau}
            \Bigp{
                \Bignorm{x^{k+2} - x^{k+1}}^{2}
                - \Bignorm{K\Bigp{x^{k+2} - x^{k+1}}}^{2}
            } \\
            &-
            \frac{\eta_{+}}{4}
            \Bignorm{
                \frac{1}{\sqrt{\tau}}K\Bigp{x^{k+2} - x^{k+1}}
                + \frac{1}{\sqrt{\sigma}}\Bigp{y^{k+1} - y^{k}}
            }^{2} \\
            &-
            \frac{\eta_{-}}{4}
            \Bignorm{
                \frac{1}{\sqrt{\tau}}K\Bigp{x^{k+2} - x^{k+1}}
                - \frac{1}{\sqrt{\sigma}}\Bigp{y^{k+1} - y^{k}}
            }^{2}.
        \end{aligned}
    \end{equation}
    \endgroup
    Moreover, if the last inequality in \eqref{eq:lyapunov:parameter_condition} is strict, then the inequality in \eqref{eq:eta_coefficients_lyapunov_inequality:pos} is strict.
\end{proposition}

\begin{proof}
    By the characterization of the proximal operator in \eqref{eq:prox_characterization}, \eqref{eq:chambolle_pock_iteration} is equivalent to
    \begin{align*}
        \frac{1}{\tau}x^{k} - L^{*}y^{k} - \frac{1}{\tau}x^{k+1} &\in \partial f\Bigp{x^{k+1}}, \\
        - \theta Lx^{k} + \frac{1}{\sigma}y^{k} + (1+\theta)Lx^{k+1} - \frac{1}{\sigma}y^{k+1} &\in \partial g^*\Bigp{y^{k+1}}.
    \end{align*}
    Hence, by the definition of the subdifferential in \eqref{eq:subdifferential_definition},
    \begingroup
    \allowdisplaybreaks
    \begin{align}
        &\mathcal{V}(k+1) - \mathcal{V}(k)
        + \mathcal{D}_{x^{\star},y^{\star}}\Bigp{x^{k+1},y^{k+1}} \notag\\
        {}&{}\leq
        \begin{aligned}[t]
            &\mathcal{V}(k+1) - \mathcal{V}(k)
            + \mathcal{D}_{x^{\star},y^{\star}}\Bigp{x^{k+1},y^{k+1}} \\
            &+ \frac{1-\theta}{2}
            \underbracket{
                \Bigp{
                    f\Bigp{x^{k+2}} - f\Bigp{x^{k+1}}
                    - \Biginner{\frac{1}{\tau}x^{k} - L^{*}y^{k} - \frac{1}{\tau}x^{k+1}}{x^{k+2} - x^{k+1}}
                }
            }_{\ge 0} \\
            &+ \underbracket{
                \Bigp{
                    f\Bigp{x^{\star}} - f\Bigp{x^{k+1}}
                    - \Biginner{\frac{1}{\tau}x^{k} - L^{*}y^{k} - \frac{1}{\tau}x^{k+1}}{x^{\star} - x^{k+1}}
                }
            }_{\ge 0} \\
            &+ \frac{1-\theta}{2}
            \underbracket{
                \Bigp{
                    g^*\Bigp{y^{k+2}} - g^*\Bigp{y^{k+1}}
                    - \Biginner{-\theta Lx^{k} + \frac{1}{\sigma}y^{k} + (1+\theta)Lx^{k+1} - \frac{1}{\sigma}y^{k+1}}{y^{k+2} - y^{k+1}}
                }
            }_{\ge 0} \\
            &+ \underbracket{
                \Bigp{
                    g^*\Bigp{y^{\star}} - g^*\Bigp{y^{k+1}}
                    - \Biginner{-\theta Lx^{k} + \frac{1}{\sigma}y^{k} + (1+\theta)Lx^{k+1} - \frac{1}{\sigma}y^{k+1}}{y^{\star} - y^{k+1}}
                }
            }_{\ge 0}
        \end{aligned} \notag\\
        {}&{}=
        -\frac{1}{4}\left(\vphantom{\frac{1}{\tau}\Bignorm{x^{k+2} - x^{k+1}}^{2}}\right.
            \frac{1}{\tau}\Bignorm{x^{k+1} - x^{k}}^{2}
            + \frac{1}{\tau}\Bignorm{x^{k+2} - x^{k+1}}^{2}
            + \frac{1}{\sigma}\Bignorm{y^{k+1} - y^{k}}^{2}
            + \frac{1}{\sigma}\Bignorm{y^{k+2} - y^{k+1}}^{2} \notag\\
            &\quad
            - \frac{2(1-\theta)}{\tau}\Biginner{x^{k+1} - x^{k}}{x^{k+2} - x^{k+1}}
            - \frac{2(1-\theta)}{\sigma}\Biginner{y^{k+1} - y^{k}}{y^{k+2} - y^{k+1}} \notag\\
            &\quad
            - (1+\theta)\Biginner{L\Bigp{x^{k+1} - x^{k}}}{y^{k+1} - y^{k}}
            + 2(1-\theta)\Biginner{L\Bigp{x^{k+2} - x^{k+1}}}{y^{k+1} - y^{k}} \notag\\
            &\quad
            - (1+\theta)\Biginner{L\Bigp{x^{k+2} - x^{k+1}}}{y^{k+2} - y^{k+1}} \notag\\
            &\quad
            + 2\theta(1-\theta)\Biginner{L\Bigp{x^{k+1} - x^{k}}}{y^{k+2} - y^{k+1}}
        \left.\vphantom{\frac{1}{\tau}\Bignorm{x^{k+2} - x^{k+1}}^{2}}\right).\label{eq:c1_unnormalized_lyapunov_inequality}
    \end{align}
    \endgroup
    Here, the first inequality uses only that \(0<\theta \leq 1\), so each underbraced term can be added without decreasing the right-hand side. The last equality is obtained by expanding \(\mathcal{V}(k+1) - \mathcal{V}(k)\) from \eqref{eq:lyapunov_function}, expanding the duality gap from \eqref{eq:pd_gap_fcn}, using \(\Biginner{L^{*}y}{x} = \Biginner{y}{Lx}\), and then collecting the increments \(x^{k+1} - x^{k}\), \(x^{k+2} - x^{k+1}\), \(y^{k+1} - y^{k}\), and \(y^{k+2} - y^{k+1}\).

    Set
    \begin{align*}
        \widehat{\delta}_x^{k} = \frac{x^{k+1} - x^{k}}{\sqrt{\tau}},
        \qquad
        \widehat{\delta}_x^{k+1} = \frac{x^{k+2} - x^{k+1}}{\sqrt{\tau}},
        \qquad
        \delta_y^{k} = \frac{y^{k+1} - y^{k}}{\sqrt{\sigma}},
        \qquad
        \delta_y^{k+1} = \frac{y^{k+2} - y^{k+1}}{\sqrt{\sigma}},
    \end{align*}
    and
    \begin{align}\label{eq:delta_x}
        \delta_x^{k} = K\widehat{\delta}_x^{k},
        \qquad
        \delta_x^{k+1} = K\widehat{\delta}_x^{k+1}.
    \end{align}
    Thus, using \(L = \Bignorm{L}K\), \eqref{eq:c1_unnormalized_lyapunov_inequality} becomes
    \begin{align}
        {}&{}\mathcal{V}(k+1) - \mathcal{V}(k)
        + \mathcal{D}_{x^{\star},y^{\star}}\Bigp{x^{k+1},y^{k+1}} \notag\\
        {}&{}\le{}
        -\frac{1}{4}\left(\vphantom{\sqrt{\tau\sigma}\Bignorm{L}\Biginner{\delta_x^{k+1}}{\delta_y^{k+1}}}\right.
            \Bignorm{\widehat{\delta}_x^{k}}^{2}
            + \Bignorm{\widehat{\delta}_x^{k+1}}^{2}
            - 2(1-\theta)\Biginner{\widehat{\delta}_x^{k}}{\widehat{\delta}_x^{k+1}}
            + \Bignorm{\delta_y^{k}}^{2}
            + \Bignorm{\delta_y^{k+1}}^{2}
            - 2(1-\theta)\Biginner{\delta_y^{k}}{\delta_y^{k+1}} \notag\\
            {}&{}
            - (1+\theta)\sqrt{\tau\sigma}\Bignorm{L}\Biginner{\delta_x^{k}}{\delta_y^{k}}
            + 2(1-\theta)\sqrt{\tau\sigma}\Bignorm{L}\Biginner{\delta_x^{k+1}}{\delta_y^{k}} \notag\\
            {}&{}
            - (1+\theta)\sqrt{\tau\sigma}\Bignorm{L}\Biginner{\delta_x^{k+1}}{\delta_y^{k+1}}
            + 2\theta(1-\theta)\sqrt{\tau\sigma}\Bignorm{L}\Biginner{\delta_x^{k}}{\delta_y^{k+1}}
        \left.\vphantom{\sqrt{\tau\sigma}\Bignorm{L}\Biginner{\delta_x^{k+1}}{\delta_y^{k+1}}}\right).\label{eq:c1_normalized_lyapunov_inequality}
    \end{align}

    Next, define
    \begin{align*}
        \mathcal{E}_k
        ={}&
        \Bigp{
            \Bignorm{\widehat{\delta}_x^{k}}^{2}
            + \Bignorm{\widehat{\delta}_x^{k+1}}^{2}
            - 2(1-\theta)\Biginner{\widehat{\delta}_x^{k}}{\widehat{\delta}_x^{k+1}}
        }
        -
        \Bigp{
            \Bignorm{\delta_x^{k}}^{2}
            + \Bignorm{\delta_x^{k+1}}^{2}
            - 2(1-\theta)\Biginner{\delta_x^{k}}{\delta_x^{k+1}}
        }.
    \end{align*}

    Then \eqref{eq:c1_normalized_lyapunov_inequality} gives
    \begin{align}
        &\mathcal{V}(k+1) - \mathcal{V}(k) + \mathcal{D}_{x^{\star},y^{\star}}\Bigp{x^{k+1},y^{k+1}} \notag \\
        {}&{}\le
        -\frac{1}{4}\mathcal{E}_k
        -\frac{1}{4}\left(\vphantom{\sqrt{\tau\sigma}\Bignorm{L}\Biginner{\delta_x^{k+1}}{\delta_y^{k+1}}}\right.
            \Bignorm{\delta_x^{k}}^{2}
            + \Bignorm{\delta_x^{k+1}}^{2}
            + \Bignorm{\delta_y^{k}}^{2}
            + \Bignorm{\delta_y^{k+1}}^{2} \notag\\
            {}&{}\quad
            - 2(1-\theta)\Biginner{\delta_x^{k}}{\delta_x^{k+1}}
            - 2(1-\theta)\Biginner{\delta_y^{k}}{\delta_y^{k+1}} \notag\\
            {}&{}\quad
            - (1+\theta)\sqrt{\tau\sigma}\Bignorm{L}\Biginner{\delta_x^{k}}{\delta_y^{k}}
            + 2(1-\theta)\sqrt{\tau\sigma}\Bignorm{L}\Biginner{\delta_x^{k+1}}{\delta_y^{k}} \notag\\
            {}&{}\quad
            - (1+\theta)\sqrt{\tau\sigma}\Bignorm{L}\Biginner{\delta_x^{k+1}}{\delta_y^{k+1}}
            + 2\theta(1-\theta)\sqrt{\tau\sigma}\Bignorm{L}\Biginner{\delta_x^{k}}{\delta_y^{k+1}}
        \left.\vphantom{\sqrt{\tau\sigma}\Bignorm{L}\Biginner{\delta_x^{k+1}}{\delta_y^{k+1}}}\right).\label{eq:c1_reduced_lyapunov_inequality}
    \end{align}

    Next, we rewrite the remaining quadratic form in terms of the sum and difference variables.
    Using
    \begingroup
    \allowdisplaybreaks
    \begin{align*}
        \Bignorm{\delta_x^{k}}^{2} + \Bignorm{\delta_y^{k+1}}^{2}
        &={}
        \frac{1}{2}\Bignorm{\delta_x^{k} + \delta_y^{k+1}}^{2}
        + \frac{1}{2}\Bignorm{\delta_x^{k} - \delta_y^{k+1}}^{2}, \\
        2\Biginner{\delta_x^{k}}{\delta_y^{k+1}}
        &={}
        \frac{1}{2}\Bignorm{\delta_x^{k} + \delta_y^{k+1}}^{2}
        - \frac{1}{2}\Bignorm{\delta_x^{k} - \delta_y^{k+1}}^{2}, \\
        \Bignorm{\delta_x^{k+1}}^{2} + \Bignorm{\delta_y^{k}}^{2}
        &={}
        \frac{1}{2}\Bignorm{\delta_x^{k+1} + \delta_y^{k}}^{2}
        + \frac{1}{2}\Bignorm{\delta_x^{k+1} - \delta_y^{k}}^{2}, \\
        2\Biginner{\delta_x^{k+1}}{\delta_y^{k}}
        &={}
        \frac{1}{2}\Bignorm{\delta_x^{k+1} + \delta_y^{k}}^{2}
        - \frac{1}{2}\Bignorm{\delta_x^{k+1} - \delta_y^{k}}^{2}, \\
        2\Biginner{\delta_x^{k}}{\delta_x^{k+1}}
        + 2\Biginner{\delta_y^{k+1}}{\delta_y^{k}}
        &={}
        \Biginner{\delta_x^{k} + \delta_y^{k+1}}{\delta_x^{k+1} + \delta_y^{k}}
        + \Biginner{\delta_x^{k} - \delta_y^{k+1}}{\delta_x^{k+1} - \delta_y^{k}}, \\
        2\Biginner{\delta_x^{k}}{\delta_y^{k}}
        + 2\Biginner{\delta_y^{k+1}}{\delta_x^{k+1}}
        &={}
        \Biginner{\delta_x^{k} + \delta_y^{k+1}}{\delta_x^{k+1} + \delta_y^{k}}
        - \Biginner{\delta_x^{k} - \delta_y^{k+1}}{\delta_x^{k+1} - \delta_y^{k}},
    \end{align*}
    \endgroup
    in \eqref{eq:c1_reduced_lyapunov_inequality} gives
    \begin{align}
        {}&{}\mathcal{V}(k+1) - \mathcal{V}(k)
        + \mathcal{D}_{x^{\star},y^{\star}}\Bigp{x^{k+1},y^{k+1}} \notag\\
        {}&{}\le
        -\frac{1}{4}\mathcal{E}_k
        -\frac{1}{4}\left(\vphantom{\gamma_{+}\Bignorm{\delta_x^{k+1} + \delta_y^{k}}^{2}}\right.
            \alpha_{+}\Bignorm{\delta_x^{k} + \delta_y^{k+1}}^{2}
            - 2\beta_{+}\Biginner{\delta_x^{k} + \delta_y^{k+1}}{\delta_x^{k+1} + \delta_y^{k}}
            + \gamma_{+}\Bignorm{\delta_x^{k+1} + \delta_y^{k}}^{2} \notag\\
            &\quad
            + \alpha_{-}\Bignorm{\delta_x^{k} - \delta_y^{k+1}}^{2}
            - 2\beta_{-}\Biginner{\delta_x^{k} - \delta_y^{k+1}}{\delta_x^{k+1} - \delta_y^{k}}
            + \gamma_{-}\Bignorm{\delta_x^{k+1} - \delta_y^{k}}^{2}
        \left.\vphantom{\gamma_{+}\Bignorm{\delta_x^{k+1} + \delta_y^{k}}^{2}}\right),\label{eq:c1_compact_lyapunov_inequality}
    \end{align}
    where
    \begin{align*}
        \alpha_{\pm} &= \frac{1 \pm \sqrt{\tau\sigma}\Bignorm{L}\theta(1-\theta)}{2}, \\
        \beta_{\pm} &= \frac{2(1-\theta) \pm (1+\theta)\sqrt{\tau\sigma}\Bignorm{L}}{4}, \\
        \gamma_{\pm} &= \frac{1 \pm \sqrt{\tau\sigma}\Bignorm{L}(1-\theta)}{2}.
    \end{align*}

    Completing the square in the \(+\) and \(-\) blocks gives
    \begin{align}
        {}&{}\mathcal{V}(k+1) - \mathcal{V}(k)
        + \mathcal{D}_{x^{\star},y^{\star}}\Bigp{x^{k+1},y^{k+1}} \notag\\
        {}&{}\le
        -\frac{1}{4}\mathcal{E}_k -\frac{1}{4}\left(\vphantom{\Bigp{\gamma_{+} - \frac{\beta_{+}^{2}}{\alpha_{+}}}\Bignorm{\delta_x^{k+1} + \delta_y^{k}}^{2}}\right.
            \alpha_{+}
            \Bignorm{
                \delta_x^{k} + \delta_y^{k+1}
                - \frac{\beta_{+}}{\alpha_{+}}\Bigp{\delta_x^{k+1} + \delta_y^{k}}
            }^{2}
            + \Bigp{\gamma_{+} - \frac{\beta_{+}^{2}}{\alpha_{+}}}
            \Bignorm{\delta_x^{k+1} + \delta_y^{k}}^{2} \notag\\
            &\quad
            + \alpha_{-}
            \Bignorm{
                \delta_x^{k} - \delta_y^{k+1}
                - \frac{\beta_{-}}{\alpha_{-}}\Bigp{\delta_x^{k+1} - \delta_y^{k}}
            }^{2}
            + \Bigp{\gamma_{-} - \frac{\beta_{-}^{2}}{\alpha_{-}}}
            \Bignorm{\delta_x^{k+1} - \delta_y^{k}}^{2}
        \left.\vphantom{\Bigp{\gamma_{+} - \frac{\beta_{+}^{2}}{\alpha_{+}}}\Bignorm{\delta_x^{k+1} + \delta_y^{k}}^{2}}\right).\label{eq:c1_completed_square_lyapunov_inequality}
    \end{align}

    Note that \eqref{eq:inequality_lemma:2} of \Cref{lem:parameters} gives \(\sqrt{\tau\sigma}\Bignorm{L}(1+\theta)\le 2\).
    Thus,
    \begin{align*}
        \alpha_{+} > 0,
        \qquad
        \alpha_{-}
        \ge
        \frac{1 - 2\theta(1-\theta)}{2}
        > 0,
    \end{align*}
    since \(0 < \theta \le 1\).
    Moreover, note that \eqref{eq:eta_coefficients_lyapunov_inequality} is exactly the identity
    \begin{align*}
        \eta_{\pm} = \gamma_{\pm} - \frac{\beta_{\pm}^{2}}{\alpha_{\pm}}.
    \end{align*}
    Therefore, \eqref{eq:c1_completed_square_lyapunov_inequality} gives
    \begin{align}
        \mathcal{V}(k+1)
        \le
        \mathcal{V}(k)
        - \mathcal{D}_{x^{\star},y^{\star}}\Bigp{x^{k+1},y^{k+1}}
        - \frac{1}{4}\mathcal{E}_k
        - \frac{\eta_{+}}{4}\Bignorm{\delta_x^{k+1} + \delta_y^{k}}^{2}
        - \frac{\eta_{-}}{4}\Bignorm{\delta_x^{k+1} - \delta_y^{k}}^{2}.\label{eq:descent_with_eta_lyapunov_inequality}
    \end{align}

    Since \(\Bignorm{K}\le 1\), \eqref{eq:delta_x} gives
    \begingroup
    \allowdisplaybreaks
    \begin{align}
        \mathcal{E}_k
        ={}&{}
        \theta\Bigp{
            \Bignorm{\widehat{\delta}_x^{k}}^{2} - \Bignorm{\delta_x^{k}}^{2}
        }
        +
        (1-\theta)\Bigp{
            \Bignorm{\widehat{\delta}_x^{k} - \widehat{\delta}_x^{k+1}}^{2}
            - \Bignorm{\delta_x^{k} - \delta_x^{k+1}}^{2}
        }
        \notag\\
        {}&{}
        +
        \theta\Bigp{
            \Bignorm{\widehat{\delta}_x^{k+1}}^{2} - \Bignorm{\delta_x^{k+1}}^{2}
        }
        \notag\\
        \geq{}&{}
        \theta\Bigp{
            \Bignorm{\widehat{\delta}_x^{k}}^{2} - \Bignorm{K}^2\Bignorm{\widehat{\delta}_x^{k}}^{2}
        }
        +
        (1-\theta)\Bigp{
            \Bignorm{\widehat{\delta}_x^{k} - \widehat{\delta}_x^{k+1}}^{2}
            - \Bignorm{K}^2\Bignorm{\widehat{\delta}_x^{k} - \widehat{\delta}_x^{k+1}}^{2}
        }
        \notag\\
        {}&{}
        +
        \theta\Bigp{
            \Bignorm{\widehat{\delta}_x^{k+1}}^{2} - \Bignorm{\delta_x^{k+1}}^{2}
        }
        \notag\\
        \geq{}&{}
        \theta\Bigp{
            \Bignorm{\widehat{\delta}_x^{k+1}}^{2}
            - \Bignorm{\delta_x^{k+1}}^{2}
        }.\label{eq:E_k_lower_bound_lyapunov_inequality}
    \end{align}
    \endgroup
    Inserting the lower bound \eqref{eq:E_k_lower_bound_lyapunov_inequality} into \eqref{eq:descent_with_eta_lyapunov_inequality} gives \eqref{eq:lyapunov_inequality}.

    Under the parameter condition \eqref{eq:lyapunov:parameter_condition}, the numerator of \(\eta_{\pm}\) in \eqref{eq:eta_coefficients_lyapunov_inequality} is nonnegative by \eqref{eq:inequality_lemma:1} of \Cref{lem:parameters}, and the denominators are positive since
    \begin{align*}
        8\Bigp{1 \pm \sqrt{\tau\sigma}\Bignorm{L}\theta(1-\theta)} = 16\alpha_{\pm} > 0,
    \end{align*}
    so \(\eta_{\pm}\ge 0\).
    Finally, if the last inequality in \eqref{eq:lyapunov:parameter_condition} is strict, then the numerator of \(\eta_{\pm}\) in \eqref{eq:eta_coefficients_lyapunov_inequality} is positive by \eqref{eq:inequality_lemma:1} of \Cref{lem:parameters}, so \(\eta_{\pm}> 0\).
\end{proof}

\section[Proof of Theorem \ref{thm:main:sequence_convergence}]{Proof of \Cref{thm:main:sequence_convergence}}\label{sec:proof_sequence_convergence}

The parameter condition \eqref{eq:main:sequence_convergence:parameter_condition} implies \eqref{eq:lyapunov:parameter_condition} and makes its last inequality strict; the proof below uses only these facts. Hence, the strict conclusions of \Cref{lem:parameters,lem:P_properties,prop:lyapunov_inequality} apply.
Let
\begin{align*}
    Z
    =
    \Bigset{(x,y)\in\calH\times\calG \xmiddle| -L^{*}y\in\partial f(x)\ \text{and}\ Lx\in\partial g^{*}(y)}.
\end{align*}
By \Cref{ass:KKT}, the set \(Z\) is nonempty.
Fix an arbitrary point \(\p{x^{\star},y^{\star}}\in Z\), let \(\mathcal{D}_{x^{\star},y^{\star}}\) be the duality gap function from \eqref{eq:pd_gap_fcn}, and let \(\mathcal{V}\) be the Lyapunov function \eqref{eq:lyapunov_function} from \Cref{def:lyapunov_function}.
\Cref{prop:lyapunov_lower_bound,prop:lyapunov_inequality} imply that the sequence \(\Bigp{\mathcal{V}(k)}_{k\in\naturals}\) is bounded below and nonincreasing, respectively.
Therefore, \(\Bigp{\mathcal{V}(k)}_{k\in\naturals}\) converges by the monotone convergence theorem.
Moreover, by \Cref{prop:lyapunov_lower_bound},
\begin{align*}
    \Bigp{\forall k\in\naturals}\quad
    0 \leq \frac{1}{2}
    \Bignorm{\Bigp{x^{k+1} - x^{\star}, y^{k+1} - y^{\star}}}_{P}^{2}
    \le
    \mathcal{V}(0),
\end{align*}
so \(\Bigp{\Bigp{x^{k},y^{k}}}_{k\in\naturals}\) is bounded with respect to \(\Bignorm{\cdot}_{P}\). Since \Cref{lem:P_properties:strict} states that \(\Bignorm{\cdot}_{P}\) is equivalent to the canonical product norm on \(\calH\times\calG\), the sequence is also bounded in \(\calH\times\calG\).

Next, \Cref{prop:lyapunov_lower_bound,prop:lyapunov_inequality} imply that
\begin{align}
    &\sum_{k=0}^{\infty} \mathcal{D}_{x^{\star},y^{\star}}\Bigp{x^{k+1},y^{k+1}} < +\infty, \label{eq:summable_1}\\
    &\sum_{k=0}^{\infty} \Bigp{ \Bignorm{x^{k+2} - x^{k+1}}^{2} - \Bignorm{K\Bigp{x^{k+2} - x^{k+1}}}^{2}} < +\infty, \label{eq:summable_2}\\
    &\sum_{k=0}^{\infty} \Bignorm{\frac{1}{\sqrt{\tau}}K\Bigp{x^{k+2} - x^{k+1}} + \frac{1}{\sqrt{\sigma}}\Bigp{y^{k+1} - y^{k}}}^{2} < +\infty, \label{eq:summable_3}\\
    &\sum_{k=0}^{\infty} \Bignorm{\frac{1}{\sqrt{\tau}}K\Bigp{x^{k+2} - x^{k+1}} - \frac{1}{\sqrt{\sigma}}\Bigp{y^{k+1} - y^{k}}}^{2} < +\infty, \label{eq:summable_4}
\end{align}
via a telescoping summation argument. Therefore, \eqref{eq:summable_1} implies that
\begin{align}
    \mathcal{D}_{x^{\star},y^{\star}}\Bigp{x^{k+1},y^{k+1}} \to 0 \text{ as }k\to\infty, \label{eq:D_zero}
\end{align}
\eqref{eq:summable_3}, \eqref{eq:summable_4}, and \(L=\Bignorm{L}K\) imply that
\begin{align}
    &L\Bigp{x^{k+1} - x^{k}} \to 0 \text{ as }k\to\infty, \label{eq:Lx_diff_to_zero}\\
    &y^{k+1} - y^{k} \to 0 \text{ as }k\to\infty,\label{eq:y_diff_to_zero}
\end{align}
and \eqref{eq:summable_2}, together with \(\Bignorm{K}\leq 1\) and \eqref{eq:Lx_diff_to_zero}, imply that
\begin{align}
    x^{k+1} - x^{k} \to 0 \text{ as }k\to\infty.\label{eq:x_diff_to_zero}
\end{align}

Note that \eqref{eq:lyapunov_function} can be written as
\begingroup
\allowdisplaybreaks
\begin{align*}
    \Bignorm{\Bigp{x^{k} - x^{\star}, y^{k} - y^{\star}}}_{P}^{2}
    ={}&{}
    2\mathcal{V}(k) \\
    {}&{} +
    \frac{1}{2}\Bignorm{\Bigp{x^{k+1} - x^{k}, y^{k+1} - y^{k}}}_{P}^{2} \\
    {}&{} +
    \Bigp{1-\theta}\mathcal{D}_{x^{\star},y^{\star}}\p{x^{k+1},y^{k+1}} \\
    {}&{} +
    \Bigp{1-\theta}
    \Bigp{
        \Biginner{y^{k} - y^{\star}}{L\p{x^{k+1} - x^{k}}}
        - \Biginner{L\p{x^{k} - x^{\star}}}{y^{k+1} - y^{k}}
    }.
\end{align*}
\endgroup
Here the first term on the right-hand side converges because \(\Bigp{\mathcal{V}(k)}_{k\in\naturals}\) converges.
The second and third terms converge to zero by \eqref{eq:y_diff_to_zero}, \eqref{eq:x_diff_to_zero}, \Cref{lem:P_properties:strict}, and \eqref{eq:D_zero}.
The last term converges to zero because \(\Bigp{\Bigp{x^{k},y^{k}}}_{k\in\naturals}\) is bounded, while \eqref{eq:Lx_diff_to_zero} and \eqref{eq:y_diff_to_zero} show that the increments vanish.
In particular,
\begin{align}\label{eq:Opial_norm_converges}
    \Bigp{\Bignorm{\Bigp{x^{k} - x^{\star}, y^{k} - y^{\star}}}_{P}^{2} }_{k\in\naturals}
\end{align}
converges.

Consider the operator \(A: \calH\times\calG \to 2^{\calH\times\calG}\) given by
\begin{align*}
    A\Bigp{x,y} = \partial f(x)\times\partial g^{*}(y) + \Bigp{L^{*}y,-Lx}.
\end{align*}
Since \(f\) and \(g^{*}\) are proper, convex, and lower semicontinuous, \(\partial f\) and \(\partial g^{*}\) are maximally monotone by \cite[Theorem~20.25]{bauschke_combettes_2017_convex_analysis_monotone}.
It follows from \cite[Proposition~20.23]{bauschke_combettes_2017_convex_analysis_monotone} that
\begin{align*}
    (x,y) \mapsto \partial f(x)\times\partial g^{*}(y)
\end{align*}
is maximally monotone on \(\calH\times\calG\).
On the other hand,
\begin{align*}
    \Bigp{x,y} \mapsto \Bigp{L^{*}y,-Lx}
\end{align*}
is maximally monotone on \(\calH\times\calG\) by \cite[Example 20.35]{bauschke_combettes_2017_convex_analysis_monotone} and has full domain.
Consequently, \(A\) is maximally monotone by \cite[Corollary 25.5]{bauschke_combettes_2017_convex_analysis_monotone}.

The bounded sequence \(\Bigp{\Bigp{x^{k},y^{k}}}_{k\in\naturals}\) has a weakly convergent subsequence.

Let \(\Bigp{\Bigp{x^{k_{n}},y^{k_{n}}}}_{n\in\naturals}\) be such a subsequence and suppose that it converges weakly to \(\Bigp{\bar x,\bar y}\in\calH\times\calG\).
Since \eqref{eq:x_diff_to_zero} and \eqref{eq:y_diff_to_zero} hold, we also have
\begin{align*}
    \Bigp{x^{k_{n}+1},y^{k_{n}+1}} \rightharpoonup \Bigp{\bar x,\bar y}.
\end{align*}
By the characterization of the proximal operator in \eqref{eq:prox_characterization} and the update rule in \eqref{eq:chambolle_pock_iteration}, we obtain
\begin{align*}
    \partial f\p{x^{k_n+1}} + L^{*}y^{k_n+1} &\ni \frac{1}{\tau}\p{x^{k_n} - x^{k_n+1}} + L^{*}\p{y^{k_n+1} - y^{k_n}} \eqcolon u^{n} \xrightarrow[n\to \infty]{} 0, \\
    \partial g^{*}\p{y^{k_n+1}} - Lx^{k_n+1} &\ni \frac{1}{\sigma}\p{y^{k_n} - y^{k_n+1}} + \theta L\p{x^{k_n+1} - x^{k_n}} \eqcolon v^{n} \xrightarrow[n\to \infty]{} 0,
\end{align*}
due to \eqref{eq:Lx_diff_to_zero}, \eqref{eq:y_diff_to_zero}, and \eqref{eq:x_diff_to_zero}.
In particular,
\begin{align*}
    A\Bigp{x^{k_{n}+1},y^{k_{n}+1}} \ni \Bigp{u^{n},v^{n}} \to \Bigp{0,0} \text{ as } n\to \infty,
\end{align*}
and we conclude, using weak-strong closedness of maximally monotone operators \cite[Proposition~20.38(ii)]{bauschke_combettes_2017_convex_analysis_monotone}, that
\begin{align}\label{eq:Opial_weak_sequential_cluster_point_in_set}
    0 \in A\Bigp{\bar x,\bar y}\quad \iff \quad \Bigp{\bar x,\bar y} \in Z,
\end{align}
i.e., every weak sequential cluster point of \(\Bigp{\Bigp{x^{k},y^{k}}}_{k\in\naturals}\) belongs to \(Z\).

Finally, because \(\p{x^{\star},y^{\star}}\in Z\) was arbitrary and both \eqref{eq:Opial_norm_converges} and \eqref{eq:Opial_weak_sequential_cluster_point_in_set} hold, \cite[Lemma~2.47]{bauschke_combettes_2017_convex_analysis_monotone} gives a point \(\p{\bar x,\bar y}\in Z\) such that
\begin{align*}
    \p{x^{k},y^{k}} \rightharpoonup \p{\bar x,\bar y} \text{ in } \Bigp{\calH\times\calG,\Biginner{\cdot}{\cdot}_{P}},
\end{align*}
and \Cref{lem:P_properties:strict} gives
\begin{align*}
    \p{x^{k},y^{k}} \rightharpoonup \p{\bar x,\bar y} \text{ in } \Bigp{\calH\times\calG,\Biginner{\cdot}{\cdot}},
\end{align*}
as claimed.

\section{Conclusion and future work}\label{sec:conclusion}

In this paper, we established weak sequential convergence of the Chambolle--Pock iterates in real Hilbert spaces for the range \(0 < \theta \leq 1/2\) under
\begin{align*}
    \tau\sigma\Bignorm{L}^{2}
    <
    \frac{4\theta(2-\theta)}{1 - 2\theta + 9\theta^{2} - 4\theta^{3}},
\end{align*}
as stated in \Cref{thm:main:sequence_convergence}.
Weak convergence had not previously been established for the Chambolle--Pock iteration over this range.
The underlying Lyapunov estimates remain valid under \eqref{eq:lyapunov:parameter_condition} throughout the interval \(0 < \theta \leq 1\).

A natural next question is whether the admissible step-size bound can be improved in the small-\(\theta\) regime.
The Lyapunov function \(\mathcal{V}\) in \eqref{eq:lyapunov_function} from \Cref{def:lyapunov_function} depends only on the states
\begin{align*}
    x^{\star}, x^{k}, x^{k+1}, y^{\star}, y^{k}, y^{k+1}.
\end{align*}
However, strong numerical evidence in \cite[Figure~2]{upadhyaya_etal_2026_auto_lyap_software_suite} suggests that one can substantially enlarge the admissible range of \(\tau\sigma\Bignorm{L}^{2}\), especially for \(0 < \theta \leq 1/2\), by considering Lyapunov functions with a slightly longer memory, namely
\begin{align*}
    x^{\star}, x^{k}, \ldots, x^{k+h}, y^{\star}, y^{k}, \ldots, y^{k+h},
\end{align*}
with \(h=3\).
At present, however, no closed-form certificate of this type is known.
Deriving such a Lyapunov function, together with an explicit step-size condition stated in closed form, appears to be a natural next step.
Current numerical evidence also suggests that taking \(h>3\) does not lead to further improvements.

\paragraph{Acknowledgments.}
The author is grateful to Sebastian Banert and Pontus Giselsson for several helpful discussions on this problem.

\section*{Declarations}

\subsection*{Funding}
M. Upadhyaya is supported by the European Union (ERC grant CASPER 101162889).
This work was also funded in part by the French government, through the Agence Nationale de la Recherche, as part of the ``France 2030'' program under reference ANR-23-IACL-0008 ``PR[AI]RIE-PSAI''.
The views and opinions expressed are those of the author only and do not necessarily reflect those of the funding agencies or granting authorities, which cannot be held responsible for them.

\subsection*{Conflict of interest}
The author declares no relevant financial or non-financial interests to disclose.

\subsection*{Data availability}
No datasets were generated or analyzed for this paper.

\phantomsection
\addcontentsline{toc}{section}{References}
\bibliographystyle{abbrvurl}
\bibliography{mybib}

\end{document}